\documentclass[11pt]{article}
\usepackage{geometry} 
\usepackage{amsmath,amsthm,amssymb}                
\usepackage{graphicx, color}
\usepackage{amssymb}
\usepackage{epstopdf}
\usepackage{pdfsync}
\usepackage{mathabx}

\newtheorem{definition}{Definition}[section]
\newtheorem{lemma}[definition]{Lemma}
\newtheorem{theorem}[definition]{Theorem}

\newtheorem{remark}[definition]{Remark}
\newtheorem{example}[definition]{Example}

\def\e{\varepsilon}

\def\dxy{\,dx dy}
\def\dx{\,dx}

\title{Discrete approximation of nonlocal-gradient energies}
\author{Andrea Braides
\\ \small SISSA, via Bonomea 265, Trieste, Italy\\  \\  Andrea Causin and
Margherita Solci \\ \small
DADU, Universit\`a di Sassari\\ \small
 piazza Duomo 6, 07041 Alghero (SS), Italy}
\date{}
\begin{document} 

\maketitle

\abstract{We study a discrete approximation of functionals depending on nonlocal gradients.
The discretized functionals are proved to be coercive in classical Sobolev spaces.

\smallskip

 }

\section{Introduction}
Variational problems involving nonlocal gradients $\nabla_\rho u$ defined by
\begin{equation}\label{nogra}
\nabla_\rho u(x)=\int_{\mathbb R^d} \rho(\xi){(u(x+\xi)-u(x))\xi\over|\xi|}d\xi,
\end{equation}   
where $\rho$ is a suitable symmetric positive kernel, have been recently considered e.g.~in \cite{SS,KS}. 
In particular, Riesz kernels have been used in connection with fractional Sobolev spaces (as in \cite{CS1,BCS}), in which case, and in general in the case of singular kernels, this integral must be considered as a principal value, but this fact is not relevant for the present paper. 

Fractional-gradient integral functionals take the form
\begin{equation}
\int_{\mathbb R^d} f(\nabla_\rho u)\dx ,
\end{equation}
and boundary-value problems can be addressed on suitably defined spaces.
These energies allow to consider problems stated in a weaker form than in usual Sobolev spaces. On the other hand, 
by scaling such gradients, an approximation can be provided of classical functionals of the Calculus of Variations \cite{BMC}. More precisely, after considering scaled kernels $\rho_\e$ defined by $\rho_\e(\xi)={1\over\e^d}\rho({\xi\over\e})$, from the weak convergence of $u_\e$ to $u$, we may deduce the weak convergence of $\nabla_{\rho_\e}u_\e$ to (a multiple of) the usual weak gradient $\nabla u$, upon some boundedness conditions on the $L^p$-norm of $\nabla_{\rho_\e}u_\e$ ($1<p<\infty$). In particular arguing as in \cite{CS1} (see also \cite{MS}) we can deduce the convergence
\begin{equation}\label{conv-rho}
\lim_{\e\to 0} \int_{\mathbb R^d} |\nabla_{\rho_\e} u|^p\dx = C_\rho^p \int_{\mathbb R^d} |\nabla u|^p\dx, \quad\hbox{ where }\quad C_\rho={1\over d}\int_{\mathbb R^d} \rho(\xi)|\xi|\,d\xi,
\end{equation}
in the spirit of the celebrated paper of Bourgain et al.~\cite{BBM}, as well as the related $\Gamma$-con\-ver\-gence result. 
These results can be achieved thanks to the characterization of nonlocal gradients in distributional form \cite{CS1,CS2,SS}, which guarantees that weak limits of nonlocal gradients are nonlocal gradients of the weak limit, or even, in the case of the convergence in \eqref{conv-rho}, classical weak gradients. 

In this paper we propose a discretized approach to energies depending on nonlocal gradients as in \eqref{nogra}. 
Even though this subject has a clear connection with numerical methods in the treatment of fractional problems (see e.g.~\cite{de2,de1,DGLZ,TD1,TD2}), this work should be viewed as part of the exploration of the use of recent techniques in the analysis of discrete systems by variational methods.
In order to explain the spirit of such an approach, we can compare the convergence in \eqref{conv-rho} with the analog convergence in fractional-type Sobolev spaces shown by Bourgain et al.~\cite{BBM}, of the type
\begin{equation}\label{conv-rho-2}
\lim_{\e\to 0} \int_{\mathbb R^d\times \mathbb R^d} \rho_\e(x-y) {|u(x)- u(y)|^p\over |x-y|^p}\dxy =  \widetilde C_\rho\int_{\mathbb R^d} |\nabla u|^p\dx,
\ \hbox{ where }\  \widetilde C_\rho=\int_{\mathbb R^d} \rho(\xi)\,d\xi,
\end{equation} 
For functionals of type \eqref{conv-rho-2} a discretization approach is possible, proving their equivalence with discrete energies depending on differences $u_i-u_j$ parameterized on a cubic lattice. Such differences can be interpreted as difference quotients of some interpolation, for which the finiteness of the energy implies boundedness in some classical Sobolev space.

In the case of nonlocal gradients such an equivalence is more delicate by the possible cancellations in \eqref{nogra}.
We focus on the one-dimensional case, proposing an extension to higher dimension at the end of the paper. In order to define discrete nonlocal gradients in parallel with \eqref{nogra} it is convenient to note that, thanks to the symmetry of $\rho$, we also have 
\begin{equation}\label{nogra-2}
\nabla_\rho u(x)=\int_{\mathbb R} \rho(\xi){u(x+\xi)\,\xi\over|\xi|}d\xi=- \int_0^{+\infty}\rho(\xi)\,u(x-\xi)d\xi+
\int_0^{+\infty}\rho(\xi)\,u(x+\xi)d\xi.
\end{equation}   
With this formula in mind, if $u\colon\mathbb Z\to\mathbb R$ then we define its {\em discrete nonlocal gradient} as the function $u'_\rho\colon\mathbb Z\to\mathbb R$, whose value at $k\in\mathbb Z$ is
\begin{eqnarray}\label{discretegrad-0}
(u'_\rho)_k=
-\sum_{i=1}^M\rho_i \,u_{k+1-i}+\sum_{i=1}^M\rho_i \,u_{k+i}
,
\end{eqnarray}
where $u_k=u(k)$ and $\rho_i$ are positive values representing a discretization of the kernel $\rho$.
Note that in order to avoid considering the value of $\rho$ at $0$ we have introduced an asymmetry in this definition, which amounts to a translation of $1\over 2$. 
This is not surprising if we view $u_k$ as an average of a continuum function over the interval $[k,k+1]$, whose center is $k+{1\over 2}$.
A formally more symmetric definition would be 
\begin{eqnarray}\label{symmgrad}
(u'_\rho)_k=
-\sum_{i=1}^M\rho_i \,u_{k-i}+\sum_{i=1}^M\rho_i \,u_{k+i}
,\end{eqnarray}
but this definition will not lead to coercive energies, as shown below.

Next, we scale this definition.
If $\e>0$ and $u\colon\e \mathbb Z\to\mathbb R$, then the {\em nonlocal gradient at scale $\e$} is the 
function $u'_{\rho_\e}\colon\e\mathbb Z\to\mathbb R$ related to the discrete kernel $\rho_\e$, which is defined from $\rho$ by scaling in the same way as for continuous kernels where $\rho_\e(x)={1\over\e}\rho\bigl({x\over\e}\bigr)$. The  value of $u'_{\rho_\e}$ at $\e k$, with $k\in\mathbb Z$ is
\begin{eqnarray}\label{discretegrad-e}
(u'_{\rho_\e})_k={1\over\e}\Bigl(-\sum_{i=1}^M\rho_i u_{k+1-i}+\sum_{i=1}^M\rho_i u_{k+i}\Bigr),
\end{eqnarray}
where now $u_k= u(\e k)$.

The main result proved below is that we can improve the weak convergence of the discrete nonlocal gradients to the weak convergence of the gradients of the interpolations. That is, if $u^\e\colon\e \mathbb Z\to\mathbb R$ and $u\colon\mathbb R\to\mathbb R$ are such that the interpolations of $u^\e$ weakly converge in $L^1_{\rm loc}$ to $u$ and the interpolations of $(u^\e)'_\rho$ weakly converge in $L^2$, then indeed the interpolations of $u^\e$ weakly converge in $W^{1,2}_{\rm loc}$ to $u$. 

 The improved convergence is not a trivial fact and requires some minimal assumptions on $\rho_i$, since in \eqref{discretegrad-e} we may have cancellations due to the changing sign of the coefficients. Indeed, if we have a constant $\rho_i=\overline \rho$ then 
\begin{eqnarray}\nonumber
(u'_{\rho_\e})_k={\overline \rho\over\e}\bigl( -u_{k-M+1}-u_{k-M+2}-\ldots-u_{k-1}-u_{k}+u_{k+1}+u_{k+2}+\ldots+ u_{k+M}\bigr).
\end{eqnarray}
If $M$ is even and we take $u^\e_k=(-1)^k$ then the nonlocal gradient (at scale $\e$) of $u^\e$ is $0$, but the interpolations of $u^\e$ only converge weakly in $L^p_{\rm loc}$ and are not bounded in any Sobolev space.
The same counterexample holds  with arbitrary $\rho_i$, also not constant, if the symmetric definition of gradient \eqref{symmgrad} is used.

As an application, in a discrete-to-continuum setting \cite{ABCS}, we can consider functionals of the form
\begin{equation}
\sum_{k\in\mathbb Z}\e f\biggl({1\over\e}\Bigl(\sum_{j=1}^M (u_{k+j}-u_k)\rho_j -\sum_{j=1}^M (u_{k-j+1}-u_k)\rho_j\Bigr)\biggr),
\end{equation}
and prove their convergence with respect to the weak convergence of the interpolations in $H^1(\mathbb R)$ to $\displaystyle\int\limits f(K u')dt$ with $K=\sum\limits_{j=1}^M \rho_j(2j-1)$.

These results on discrete functions can be read in the continuum case as statements on average values of sequences with bounded energies. For example, if $\rho$ has support $[-1,1]$ and $\rho_i=\rho({i\over M})$ then we are considering a piecewise-constant approximation of $\rho$. Given a continuum $u$ and defined the value $u_k$ as the average of $u$ on $[\e {k\over M}, \e {k+1\over M}]$, the discrete nonlocal gradient of $\{u_k\}_k$ corresponds to the continuum nonlocal gradient of $u$ for the discretized kernel at $\e k$, and the result above can be read as a compactness result in $H^1_{\rm loc}$ for (the piecewise-affine interpolations of) such averages.

\section{Discrete nonlocal gradients}

We consider $M\in\mathbb N$ and a decreasing array of positive numbers $\rho_1,\ldots,\rho_M$.
Let $u\colon\mathbb Z\to\mathbb R$ and let $u_k=u(k)$. The {\em discrete non-local gradient} related to $\rho$  is 
the function $u'_\rho\colon\mathbb Z\to\mathbb R$, whose value at $k\in\mathbb Z$ is defined by 
\eqref{discretegrad-0}.
Note that we can equivalently write this quantity as
\begin{eqnarray}\label{discretegrad-diff}\nonumber
(u'_\rho)_k&=&  \rho_M(u_{k+M}-u_{k-M+1})+ \rho_{M-1} (u_{k+M-1}-u_{k-M+2})+\ldots +\rho_1(u_{k+1}-u_k)
\\ \nonumber
&=& \rho_M\sum_{j=k-M+2}^{k+M}(u_{j}-u_{j-1})+ \rho_{M-1} \sum_{j=k-M+3}^{k+M-1}(u_{j}-u_{j-1})+\ldots +\rho_1(u_{k+1}-u_k)
\\ \nonumber
&=&\rho_M (u_{k-M+2}- u_{k-M+1}) +(\rho_{M-1}+\rho_M)(u_{k-M+3}-u_{k-M+2})+\ldots
\\ \nonumber&&+(\rho_2+\rho_3+\ldots+\rho_M)(u_k-u_{k-1})+
(\rho_1+\rho_2+\ldots+\rho_M)(u_{k+1}-u_k)\\  \nonumber
&&+(\rho_2+\rho_3+\ldots+\rho_M)(u_{k+2}-u_{k+1})+ \ldots
\\
&& +(\rho_{M-1}+\rho_M)(u_{k+M-1}-u_{k+M-2})+\rho_M (u_{k+M}- u_{k+M-1}).
\end{eqnarray}

We will consider $\e>0$ and the {\em scaled discrete non-local gradients} defined for functions $u\colon\e\mathbb Z\to\mathbb R$ in \eqref{discretegrad-e} as the functions $u'_{\rho_\e}\colon\e\mathbb Z\to\mathbb R$ given at the point $\e k$ by
\begin{equation}
(u'_{\rho_\e})_k= {1\over\e}(u'_\rho)_k,
\end{equation}
where $u'_\rho$ is given by \eqref{discretegrad-0} and we have used the notation $u_k=u(\e k)$. 

 Note that if we regard the value of $u_k$ as a mean value of a continuous function $u$ over the interval $[k,k+1]$ then we loose some symmetry. In particular, the analog of formula \eqref{nogra}, obtained from \eqref{discretegrad-e} subtracting $u_k$ from all terms, reads as
\begin{eqnarray}\label{discretegrad}\nonumber
&&\hskip-2cm(u'_{\rho_\e})_k= {1\over\e}\Bigl(-\rho_M(u_{k-M+1}-u_{k})-\rho_{M-1}(u_{k-M+2}-u_{k})-\ldots-\rho_{2}(u_{k-1}-u_{k})\\
&& +\rho_1(u_{k+1}-u_k)+\ldots+\rho_{M-1}(u_{k+M-1}-u_{k}) +\rho_M(u_{k+M}-u_{k})\Bigr).
\end{eqnarray}

Even though by 
\eqref{discretegrad-diff} $(u'_{\rho_\e})_k$ can be seen as a combination of the difference quotients
$$
u_{j+1}-u_{j}\over\e
$$
for ${k-M+1}< j\le k+M-1$, due to the sign changes in \eqref{discretegrad}, in general $(u'_\rho)_k$ cannot be interpreted in terms of difference quotients of some interpolation for which a bound in some Sobolev space can be derived, except for $M=1$, in which case only one term is present and we have a classical nearest-neighbour interaction problem.

If $u_i=\varphi_i=\varphi(\e i)$ for some $C^1$-function, 
since 
$$
{u_{k+j}-u_{k+j-1}\over\e}= {\varphi(\e(k+j))-\varphi(\e(k+j-1))\over\e}= \varphi'(\e k)+o(1)
$$
as $\e\to 0$ for all $j\in\{-M+2,\ldots,M\}$, by \eqref{discretegrad-diff} 
we have 
\begin{eqnarray}\label{discretegrad-phi}\nonumber
(u'_{\rho_\e})_k=   \varphi'(\e k) \sum_{j=1}^M \rho_j (2j-1)+o(1),\end{eqnarray}
so that the piecewise-affine (or, equivalently, the piecewise-constant) interpolations of $\varphi'_{\rho_\e}$ converge to $\varphi'$ times the constant $K:=\sum_{j=1}^M \rho_j (2j-1)$. 

We will examine the asymptotic behaviour of functionals of the form
\begin{equation}
F_\e(u)=\sum_{k\in\mathbb Z} \e f\bigl((u'_{\rho_\e})_k\bigr)
\end{equation}
when $f(z)\ge c_1|z|^2$, and in particular their coerciveness properties.
To that end, let $u^\e:\e\mathbb Z\to\mathbb R$. 
Note that if $\varphi\in C^\infty_c(\mathbb R)$, from the equality
$$
\sum_{k\in\mathbb Z}\e ((u^\e)'_{\rho_\e})_k\varphi_k= -\sum_{k\in\mathbb Z}\e (\varphi'_{\rho_\e})_k u^\e_k,
$$
we deduce that if the interpolations of $\{u^\e_k\}_k$  weakly converge to some $u$ in $L^2(\mathbb R)$ as $\e\to0$ and the interpolations of $\{((u^\e)'_{\rho_\e})_k\}_k$  weakly converge to some $v$ in $L^2(\mathbb R)$  as $\e\to0$, then 
\begin{equation}\label{weakeq}
\int_{\mathbb R} v\varphi\dx= -\int_{\mathbb R} K\varphi' u\dx.
\end{equation}
Hence, $u\in W^{1,2}(\mathbb R)$ and 
the interpolations of $\{((u^\e)'_{\rho_\e})_k\}_k$ weakly converge to $Ku'$. Our aim is to improve this convergence showing that actually the piecewise-affine interpolations of $\{u^\e_k\}_k$ converge in  $W^{1,2}(\mathbb R)$.

\section{Eigenvalues of banded circulant matrices}
Using the second equality in \eqref{discretegrad-diff} we will express the discrete nonlocal gradient as a linear combination of differences of nearest neighbours through a Toeplitz matrix, or, equivalently considering boundary conditions, a circulant matrix. Coerciveness properties can be deduced from bounds on minimal eigenvalues of such a matrix, for which a general result can be proved.

\smallskip

We consider symmetric $n$-banded circulant matrices; that is, $N\times N$ matrices of the form 
\begin{equation}\label{top}
A=\begin{pmatrix}
\sigma_0 & \sigma_1 & \sigma_2& \ & \cdots & \sigma_{N-1}
\\
\sigma_{N-1}& \sigma_0 & \sigma_1 & \sigma_2&  & \vdots 
\\
\ &\sigma_{N-1}& \sigma_0 & \sigma_1 &\ddots &\ \\
\vdots & \ddots &\ddots &\ddots &\ & \sigma_2
\\
\ &\ &\ &\ &\ &\sigma_1\\
\sigma_1&\cdots& \ &\ &\sigma_{N-1}&\sigma_0
\end{pmatrix}
\end{equation}
with $\sigma_{N-j}=\sigma_j$, $\sigma_n\neq 0$ and $\sigma_j=0$ if $j\in\{n+1, \ldots, N-n-1\}$.

We assume that the following convexity condition holds 
\begin{equation}\label{convex}
\sigma_{j-1}-\sigma_{j}>\sigma_j-\sigma_{j+1} \ \ \hbox{\rm for all }\ j\in\{1, \dots, n\},
\end{equation} 
which in particular implies that $\sigma_j>0$ for $j\in\{0,\dots, n\}$.

\begin{lemma}\label{lemma} Let $N>2n$. Then $\lambda_{\rm min}$, the minimal eigenvalue of $A$, is larger than a positive constant independent of $N$ .
\end{lemma}

\begin{proof} By \cite[Chapter 3]{Gray} and the symmetry of the matrix, the minimal eigenvalue of $A$ is bounded from below by the minimum of the function
$$
\Phi(t)= \sigma_0+2\sum_{j=1}^n \sigma_j\cos(jt)
$$
for $t\in [0,\pi]$. The positivity of this trigonometric sum is a classical result due to Fej\'er (see \cite{Fejer} for the orginal source or \cite[Chapter 4]{MMR} for a review and English translation): in summary, using a closed form of Fej\'er kernels $\sum\limits_{|k|<j} (j-k) \cos(kt)={1-\cos (jt)\over 1-\cos t}$, we can rewrite $\Phi$ as
\begin{equation}\label{fejer}
\Phi(t) =\sum_{j=1}^{n+1} (\sigma_{j-1}-2\sigma_j+\sigma_{j+1}) {1-\cos (jt)\over 1-\cos t}.
\end{equation}
By \eqref{convex} each coefficient $\sigma_{j-1}-2\sigma_j+\sigma_{j+1}$ is strictly positive, so that $\Phi$ is a sum of non-negative functions. In particular
$$
\min\Phi\ge \sigma_{0}-2\sigma_1+\sigma_{2}>0
$$
and the claim.
\end{proof}

\noindent {\bf Remark} If $A=(a_{i,j})$ is a symmetric, $n-$banded, $N\times N$ Toeplitz matrix;  that is, $ a_{i,j}=\sigma_{|i-j|}$ for $|i-j|\leq n$ with $\sigma_{n}\neq 0$ and $a_{i,j}=0$ otherwise, then a general result about Hermitian Toeplitz matrices  \cite[Lemma $4.1$]{Gray} ensures that the eigenvalues of $A$ belong to the interval $[m_\Phi, M_\Phi]$ whose endpoints are respectively the minimum and the maximum of the Fourier series $$\Phi(t)=\sum_{k=-\infty}^{+\infty}\sigma_k e^{ikt}=  \sigma_0+2\sum_{k=1}^n \sigma_k\cos(kt).$$
Henceforth, if the numbers $\{\sigma_k\}_{k=0,\dots , n}$ satisfy convexity condition (\ref{convex}) then Lemma \ref{lemma} also holds for this class of Toeplitz matrices.

\section{Coerciveness and discrete-to-continuum convergence}

%
%
%
%

We first examine the coerciveness properties of reference quadratic energies as follows.

\begin{theorem}\label{quacoer} Let $\rho_i>0$, $i\in\{1,\ldots M\}$ be a decreasing array of real numbers.  Let $(a,b)$ be a bounded interval in $\mathbb R$ and let the energies
\begin{equation}
F_\e(u)=\sum_{k\in\mathbb Z}\e \biggl|{1\over\e}\Bigl(\sum_{j=1}^M (u_{k+j}-u_k)\rho_j -\sum_{j=1}^M (u_{k-j+1}-u_k)\rho_j\Bigr)\biggr|^2
\end{equation}
be defined for $u\colon\e\mathbb Z\to \mathbb R$ with $u(x)=0$ if $x\in \mathbb R\setminus (a,b)$. Then there exists a constant $\Lambda$ such that 
\begin{equation}\label{coerc-e}
F_\e(u)\ge \Lambda\sum_{k\in\mathbb Z}\e \Bigl|{u_{k+1}-u_k\over\e}\Bigr|^2
\end{equation}
for all $u$ and $\e<{1\over 2M}$.
\end{theorem}

\begin{proof} We can suppose without loss of generality that $(a,b)=(0,1)$.

Let $A$ be the matrix defined in \eqref{top} with $n=M-1$ and 
$$
\sigma_j=\sum_{k=j+1}^M \rho_k.
$$
Note that the monotonicity condition on $\rho_i$ ensures that \eqref{convex} holds, and that 
the function in \eqref{fejer} is given by
$$
\Phi(t) =\sum_{j=1}^{n+1} (\rho_j-\rho_{j+1}) {1-\cos (jt)\over 1-\cos t}.
$$

By Lemma \ref{lemma} for all $z$ with $\sum_{k=1}^N z_k^2=1$ we have
\begin{eqnarray*}
|Az|\ge |\langle Az,z\rangle|\ge \lambda_{\rm min},
\end{eqnarray*}
so that for all $z$ we have 
$
|Az|^2\ge\Lambda|z|^2
$,
where $\Lambda=(\lambda_{\rm min})^2$. Hence, \eqref{coerc-e} follows upon taking $N\ge {1\over\e}+4M$
and applying the previous estimate to $z_k= {1\over\e} (u_{k+2M}- u_{k+2M-1})$. Note that $z_k=0$ for $k\in\{1,\ldots, 2M\}$ and $k\in \{N-2M+1,\ldots, N\}$, so that
$$
F_\e(u) = \sum_{k\in\mathbb Z}\e \bigl|(A z)_{k-2M}\bigr|^2,
$$
and the claim follows.
\end{proof}

The following result proves a discrete-to-continuum convergence for discrete energies using the improved coerciveness.

\begin{theorem} Let $\rho_i>0$, $i\in\{1,\ldots M\}$ be a decreasing array of real numbers.  Let $f$ be a convex function with $c_1|z|^2+c_0\le f(z)\le c_2|z|^2+c_3$ with $c_1, c_2>0$.
Let $(a,b)$ be a bounded interval in $\mathbb R$ and let the energies
\begin{equation}
F_\e(u)=\sum_{k\in\mathbb Z}\e f\biggl({1\over\e}\Bigl(\sum_{j=1}^M (u_{k+j}-u_k)\rho_j -\sum_{j=1}^M (u_{k-j+1}-u_k)\rho_j\Bigr)\biggr)
\end{equation}
be defined for $u\colon\e\mathbb Z\to \mathbb R$ with $u(x)=0$ if $x\in \mathbb R\setminus (a,b)$. Then there exists the $\Gamma$-limit of $F_\e$ with respect to the weak $L^2$-convergence of interpolations as $\e\to0$ and 
\begin{equation}\label{lim-e}
\Gamma\hbox{-}\lim_{\e\to0}F_\e(u)= \int_{(a,b)}f(Ku')dt, \qquad K=\sum_{j=1}^M(2j-1)\rho_j.
\end{equation}
with domain $H^1_0(a,b)$.
\end{theorem}

\begin{proof} Let $u^\e$ converge weakly to $u$ and let $F_\e(u^\e)$ be equibounded. Then by the previous theorem the sequence of the corresponding piecewise-affine interpolations is weakly precompact in $H^1(\mathbb R)$,
so that indeed $u^\e$ weakly converges to $u$ in $H^1(\mathbb R)$. Since $u^\e=0$ outside $(a,b)$ the convergence is actually in $H^1_0(a,b)$. Since for all fixed $j$ all interpolations of difference quotients ${1\over\e} (u^\e_{k+j-1}- u^\e_{k+j})$ weakly converge to $u'$, by \eqref{discretegrad-diff} the weak limit of $(u^\e)'_\rho$ is $K u'$ (this can also be obtained as in \eqref{weakeq}). By the weak lower semicontinuity of $z\mapsto \int f(z)\,dt$ we then obtain the liminf inequality.

If $u\in C^\infty_c(a,b)$, extended by $0$ outside $(a,b)$, then we have
$$
F_\e(u)= \sum_{k\in \mathbb Z} \e f(K u'(\e k)) + o(1),
$$
as $\e\to 0$, and we obtain the pointwise convergence to $\int_{(a,b)}f(Ku')dt$. The limsup inequality follows by density.
\end{proof}

\section{Application to continuum interpolations}
We will use discretizations to provide an approximation for $\Gamma$-limits of continuum functionals of the form
$$
F_\e(u)=\int_\mathbb R f(\nabla_{\rho_\e}u)\dx
$$
in the one-dimensional setting, with respect to the weak convergence in $L^2$. As above, here $\rho_\e(\xi)={1\over\e}\rho({\xi\over\e})$. 

Note that the equality
$$
\int_\mathbb R \varphi(x) \nabla_{\rho_\e}u(x)\dx= - \int_\mathbb R u(x) \nabla_{\rho_\e}\varphi(x)\dx
$$
implies as in \eqref{weakeq} that if $u_\e$ weakly converges in $L^2(\mathbb R)$ to $u$ and the sequence $ \nabla_{\rho_\e}u_\e$ is bounded in $L^2(\mathbb R)$, then actually $u\in H^1(\mathbb R)$ and  the sequence $ \nabla_{\rho_\e}u_\e$ weakly converges in $L^2(\mathbb R)$ to $K u'$, where
$$
K=\int_\mathbb R \rho(\xi)|\xi|\,d\xi,
$$
and, if a growth condition of the type  $f(z)\ge c_1|z|^2$ holds, we deduce the $\Gamma$-convergence of
$F_\e$ to 
$$
F(u)=\int_\mathbb R f(Ku')\dx
$$
with respect to the weak convergence in $L^2(\mathbb R)$.
 However, we observe that for sequences of functions $u_\e\in L^2(\mathbb R)$ with $F_\e(u_\e)$ equibounded in general we cannot deduce any stronger coerciveness. Indeed, note that if $\rho$ is integrable then for each $\e>0$ $\nabla_{\rho_\e}$ is a continuous operator  in $L^2(\mathbb R)$, so that for a fixed function $u\in H^1(\mathbb R)$ we can find $u_\e$ tending to $u$ in the $L^2$-norm such that  $\nabla_{\rho_\e}u_\e$ is close to  $\nabla_{\rho_\e}u$ but with $\nabla u_\e$ unbounded in $L^2(\mathbb R)$. More in general, we can give an explicit counterexample valid also for Riesz fractional gradients. Let $\rho(\xi)={|\xi|^{-1-\alpha}}$ with $\alpha \in(0,1)$. 
Let $R>1$ be fixed and let $\varphi$ be the cut-off function defined as $\varphi(t)=\min\{1, (R-|t|)^+\}$.  
We define 
$$u_\e(t)=\e^2\sin\Big(\frac{t}{\e^2}\Big)\varphi(t).$$ 
We then have  
$$F_\e(u_\e)\leq 2\int_{-R}^R\Big(\frac{1}{\e}\int_{-\infty}^{\infty} \rho_{\e}(x-y)(u_\e(x)-u_\e(y)) \frac{x-y}{|x-y|}\, dx\, \Big)^2 dy.$$ 
By using the bounds $|u_\e(t)-u_\e(s)|\leq 2|t-s|$ and $|u_\e(t)|\leq \e^2$, 
we obtain 
\begin{eqnarray*}
&&{\hspace{-1cm}}\frac{1}{\e}\Big|\int_{-\infty}^{\infty} \rho_{\e}(x-y)(u_\e(x)-u_\e(y)) \frac{x-y}{|x-y|}\, dx\,\Big| \\
&&\leq \frac{1}{\e^2}\int_{\{|\xi|<\e^{3/2}\}} \frac{\e^{1+\alpha}}{|\xi|^{1+\alpha}}|u_\e(y+\xi)-u_\e(y)|\, d\xi\, + 
\frac{1}{\e^2}\int_{\{|\xi|>\e^{3/2}\}} \frac{\e^{1+\alpha}}{|\xi|^{1+\alpha}}|u_\e(y+\xi)|\, d\xi\, \\
&&\leq \frac{4}{\e^{1-\alpha}}\int_{0}^{\e^{3/2}} \frac{1}{\xi^{\alpha}}\, d\xi\, + 
2\e^{1+\alpha}\int_{\e^{3/2}}^{+\infty} \frac{1}{\xi^{1+\alpha}}\, d\xi 
=\frac{4}{1-\alpha} \e^{\frac{1-\alpha}{2}} + 
\frac{2}{\alpha} \e^{1-\frac{\alpha}{2}}\leq c(\alpha)\e^{\frac{1-\alpha}{2}}. 
\end{eqnarray*}
Hence, 
$$F_\e(u_\e)\leq 4R\ c(\alpha)^2\e^{1-\alpha}$$ 
which is infinitesimal as $\e\to 0$, so that  there is no constant $c$ such that  
$F_\e(u_\e)\geq c \|u_\e^\prime\|^2_{L^2}$. 

\bigskip
We suppose now that $\rho:\mathbb R\to [0,+\infty)$ be a non-negative even continuous kernel with support $[-1,1]$ and decreasing on $[0,1]$. 
For $M\in\mathbb N$ we let $\rho^M_i=\rho({i\over M})$  for $i\in\{1,\ldots, M\}$ and define the even piecewise-constant function $\rho^M$ by 
$$\rho^M(\xi)=\rho_i^M \ \ \ \hbox{\rm in }\ \Big(\frac{i-1}{M}, \frac{i}{M}\Big).$$
We also set $\rho^M_\e(\xi)=\frac{1}{\e}\rho^M(\frac{\xi}{\e})$. The family $F^{M,\e}$ of discrete energies defined on functions $u\colon {\e\over M} \mathbb Z\to\mathbb R$ by 
$$
 F^{M,\e}(u)=\sum_{k\in\mathbb Z} {\e\over M} f\bigl((u'_{\rho^M_\e})_k\bigr)
$$
can be interpreted as an approximation of the family $F_\e$ in the sense that the limit as $M\to+\infty$ of the $\Gamma$-limit of $F^M_\e$ coincides with that of $F_\e$. Moreover, for fixed $M$, the family $\{ F^{M,\e}\}_\e$ is equicoercive in $H^1(\mathbb R)$ in the sense specified in the first part of the paper. 

The sequence $F^{M,\e}$ can be related to the sequence of continuum energies
$$
F^M_\e(u)=\int_\mathbb R f(\nabla_{\rho^M_\e}u)\dx.
$$
Given a sequence  $\{u_\e\}$ with $F^M_\e(u_\e)$ equibounded, we can suppose, up to a small translation, that
$$
F^M_\e(u_\e)\ge \sum_{k\in\mathbb Z}  \frac{\e}{M} 
f\bigg(\frac{1}{\e}\int_{-\e}^\e\rho_\e(\xi) \,u\Big(\frac{\e k}{M}+\xi\Big)\frac{\xi}{|\xi|}\, d\xi\bigg).
$$
We define the sequence of discrete functions 
$u^{\e,M}\colon {\e\over M} \mathbb Z\to\mathbb R$ as 
$$
u^{\e,M}\Big(\frac{\e}{M}j\Big)=u^{\e,M}_{j}=\frac{\e}{M}\int_{\frac{\e (j-1)}{M}}^{\frac{\e j}{M}}u_\e(t)\, dt.
$$ 
Moreover, we define the functions $z^{\e,M}\colon \frac{\e}{M}\mathbb Z\to\mathbb R$ by 
$$
z^{\e,M}\Big(\frac{\e}{M}j\Big)=z^{\e,M}_j=\frac{M}{\e}(u^{\e,M}_{j}-u^{\e,M}_{j-1}).
$$
These functions are the difference quotients of $u^{\e,M}$ on ${\e\over M} \mathbb Z$.

We have \begin{eqnarray*}
&&\hskip-1.5cm\nabla_{\rho^M_\e}u\Big(\frac{\e k}{M}\Bigr)=
\frac{1}{\e}\int_{\mathbb R}\rho^M_\e(\xi) u\Big(\frac{\e k}{M}+\xi\Big)\frac{\xi}{|\xi|}\, d\xi\\
&=&\frac{1}{\e^2}\int_{-\e}^\e\rho^M\Big(\frac{\xi}{\e}\Big) u\Big(\frac{\e k}{M}+\xi\Big)\frac{\xi}{|\xi|}\, d\xi\\
&=&
\frac{1}{\e M} \Big(\sum_{i=1}^M \rho^M_i u^{\e,M}_{i+k}- \sum_{i=-M+1}^{0}\rho^M_{1-i} u^{\e,M}_{i+k}\Big)\\
&=&
\frac{1}{\e M} \Big(\rho^M_1 (u^{\e,M}_{1+k}-u^{\e,M}_{k})+\rho^M_2(u^{\e,M}_{2+k}-u^{\e,M}_{-1+k})+\dots+
\rho^M_M(u^{\e,M}_{M+k}-u^{\e,M}_{-M+1+k})\Big)\\
&=&\frac{1}{M^2}\Big( (\sum_{j=1}^M \rho^M_j) z^{\e,M}_{1+k} 
+ (\sum_{j=2}^M \rho^M_j) (z^{\e,M}_{k}+ z^{\e,M}_{2+k}) + \dots + 
\rho^M_M(z^{\e,M}_{-M+1+k}+z^{\e,M}_{M+k})\Big),
\end{eqnarray*} 
so that $\nabla_{\rho^M_\e}u$ at ${\e\over M} k$ coincides with the discrete gradient of $u^{\e,M}$ at $k$.

For $i\in \{0,\dots, M-1\}$ we consider
$$\sigma_i=\sigma_i^M=\sum_{j=i+1}^M \rho^M_j.$$ 
Let $u^\e$ be compactly supported, so that $u^\e_k$ is not zero for $k\in [-N_\e,N_\e]$, and we consider the circulating matrix as defined above of dimension $2N_\e+2M+1$, which is denoted by $A^M$.

If we take $z=z^{\e,M}$ as the vector with components $z^{\e,M}_j$, we obtain  
\begin{eqnarray*}
&&\hspace{-1cm}\sum_{k=-N_\e-M}^{N_\e+M} \frac{\e}{M} 
\Big(\frac{1}{\e}\int_{-\e}^\e\rho_\e(\xi) u\Big(\frac{\e k}{M}+\xi\Big)\frac{\xi}{|\xi|}\, d\xi\Big)^2  \\
& =&\sum_{k=-N_\e-M}^{N_\e+M} \frac{\e}{M} 
\Big( \frac{1}{M^2} \langle A^Mz, e_{k+1}\rangle \Big)^2=  \frac{\e}{M} 
 \Bigl(\frac{1}{M^2} |A^Mz|\Bigr)^2\\
 &\geq &
\frac{\e}{M} \Big(\frac{\lambda_{\rm min}}{M^2}\Big)^2 |z|^2= \Big(\frac{\lambda_{\rm min}}{M^2}\Big)^2\sum_{k=-N_\e-M}^{N_\e+M} \frac{\e}{M} \Big(\frac{u^{\e,M}_k-u^{\e,M}_{k-1}}{{\e\over M}}\Big)^2, 
\end{eqnarray*} 
which proves the equicoerciveness of $\{F^M_\e\}_\e$. Moreover, if $u_\e$ weakly converges to $u$ in $L^2(\mathbb R)$ then $\{u^{\e,M}\}_\e$ weakly converge to $u$ in $H^1(\mathbb R)$ and we have the lower bound
$$
\liminf_{\e\to 0^+}F^M_\e(u_\e) \ge \liminf_{\e\to 0^+} F^{M,\e}(u^{\e,M}).
$$
A direct computation for $u\in C^1(\mathbb R)$ gives also the upper bound and shows the equality of the $\Gamma$-limits. By letting $M\to+\infty$ we obtain an approximation of the $\Gamma$-limit of $F_\e$.

\section{Generalization to higher dimensions}
We can follow the arguments used above to provide a discrete approximation in dimension higher than one. The definition of discrete nonlocal gradient needs some care in order to avoid excessive cancellations of terms. This can be done using some slight asymmetries as for the similar problems already encountered in dimension one. 

In order to generalize the one-dimensional definition, note that we can rewrite the one-dimensional non-local gradient at a point $k\in\mathbb Z$ as
\begin{equation}
\bigl(\nabla_\rho u\bigr)_k= \sum_{i\in \mathbb Z,\ i>0} \rho_i u_{k+i}-\sum_{i\in \mathbb Z,\ i\le 0} \rho_iu_{k+i}
=\sum_{i\in \mathbb Z} \rho_i u_{k+i}\,{\rm sign}\Bigl(i-{1\over 2}\Bigr),
\end{equation}
where 
\begin{equation}
\rho_i= \rho\Bigl(|i-{1\over 2}|\Bigr).
\end{equation}
This definition could be transposed to functions defined in $\mathbb Z^d$, for which the discrete nonlocal gradient is a vector in $\mathbb R^d$. The discrete nonlocal derivative in the $n$-th direction could be defined as 
\begin{equation}
\Bigl({\partial_\rho u\over\partial x_n} \Bigr)_k= \sum_{i\in \mathbb Z^d} \rho_i u_{k+i}{i_n-{1\over 2}\over \Bigl|i-\Bigl({1\over 2},\ldots {1\over 2}\Bigr)\Bigr|}
\end{equation}
for $k\in\mathbb Z^d$, where 
\begin{equation}
\rho_i= \rho\Bigl(\Bigl|i-\Bigl({1\over 2},\ldots {1\over 2}\Bigr)\Bigr|\Bigr).
\end{equation}
However, this definition would not allow to describe properties of the interpolation of functions $u$, since, for example, the oscillating function $u$ with value $u_k=(-1)^{|k_1|+\cdots+|k_n|}$ would have zero discrete non-local gradient. 

We slightly modify the definition above introducing an asymmetry between the $n$-th direction and the others, which forbids oscillations with zero gradient.

\begin{definition}
Let $\rho:[0,+\infty)\to [0,+\infty)$ be a positive kernel with support $[0,M]$, decreasing in $[0,M]$.
We define  the {\em discrete nonlocal partial derivative} in the $n$-th direction of a function $u:\mathbb Z^d\to\mathbb R$ as the function defined by
\begin{equation}
\Bigl({\partial_\rho u\over\partial x_n} \Bigr)_k= \sum_{i\in \mathbb Z^d} \rho^n_i u_{k+i}\end{equation}
for $k\in\mathbb Z^d$, where 
\begin{equation}
\rho^n_i= \rho\Bigl(|i-{1\over 2}e_n|\Bigr){i_n-{1\over 2}\over |i-{1\over 2}e_n|}.
\end{equation}
Note that $\rho^n_i$ is non-negative for $i_n>0$ and non-positive for $i_n\le 0$, in analogy with the one-dimensional definition. The {\em discrete nonlocal gradient} is the vector $\nabla_{\rho} u$ in $\mathbb R^d$ whose $n$-th component is the discrete nonlocal derivative in the $n$-th direction.

If $\e>0$ the scaled discrete nonlocal partial derivatives $\displaystyle{\partial_{\rho_\e} u\over\partial x_n}$ and related gradient are defined by scaling as in the one-dimensional case.
\end{definition}

Properties of coerciveness for energies involving discrete nonlocal gradients can be proven by resorting to examining one-dimensional sections. We give an example in dimension two.

\begin{example}\rm In dimension $d=2$ we consider the simplest non-trivial case with $M=2$. In this case we have
\begin{eqnarray}\nonumber
\Bigl({\partial_\rho u\over\partial x_1} \Bigr)_k&=& \rho_1(u_{k+e_1}- u_k)+ \rho_2(u_{k+2e_1}- u_{k-e_1})\\
&&+ \varrho(u_{k+e_1+e_2}- u_{k+e_2})+\varrho(u_{k+e_1-e_2}- u_{k-e_2}),
\end{eqnarray}
where 
$$
\rho_1= \rho\Big({1\over 2}\Big),\qquad \rho_2=\rho\Big({3\over 2}\Big),\qquad \varrho=\rho\Big({\sqrt5\over 2}\Big){1\over\sqrt5}.
$$
Following the one-dimensional argument, we rewrite this sum, in terms of the differences $z_k=u_{k+e_1}-u_k$, as
\begin{eqnarray}\nonumber
\Bigl({\partial_\rho u\over\partial x_1} \Bigr)_k= (\rho_1+\rho_2)z_k+ \rho_2 z_{k+e_1}+\rho_2 z_{k-e_1}+\varrho z_{k+e_2}+\varrho z_{k-e_2}.
\end{eqnarray}

Suppose now that the support of $u$ be contained in $[-N+2,N-2]^2$. Then we can write 
$$
\Bigl({\partial_\rho u\over\partial x_1} \Bigr)_k= (A^N z)_k,
$$
where $A^N$ is a $N^2\times N^2$ symmetric Toeplitz  circulant matrix with $\rho_1+\rho_2$ on the diagonal, $\rho_1$ on the two next off-diagonal terms, and $\varrho$ on the $N$-th neighbours. We can then use \cite[Chapter 3]{Gray} and the symmetry of the matrix, to bound the minimal eigenvalue of $A^N$ i by the minimum of the function
$$
\Phi^N(t)= \rho_1+\rho_2+2\rho_2\cos(t)+ 2\varrho\cos(Nt).
$$
A sufficient condition independent of $N$ that ensures that the minimal eigenvalue of $A^N$ is strictly positive is 
\begin{equation}\label{varrho}
\rho_1>\rho_2+2\varrho.
\end{equation}
We can argue in the same way for the partial derivative in the $x_2$-direction.
If condition \eqref{varrho} is satisfied then we can argue as in the proof of Theorem \ref{quacoer}. Namely, if we define
\begin{equation}
F_\e(u)=\sum_{k\in\mathbb Z^2}\e^2 \bigl|(\nabla_{\rho_\e} u)_k\bigr|^2,
\end{equation}
then there exists $\Lambda$ such that 
\begin{equation}\label{2dcoer}
F_\e(u)\ge\Lambda\sum_{k,\ell\in\mathbb Z^2, |k-\ell|=1}\e^2 \biggl|{u_k-u_\ell\over\e}\biggr|^2,
\end{equation}
where $u_k= u(\e k)$ for $u:\e\mathbb Z^2\to\mathbb R$.
\end{example}

We do not pursue further the very interesting issue of the optimization of the conditions on $\rho$ to ensure coerciveness conditions as in \eqref{2dcoer}.

\bigskip
\noindent{\bf Acknowledgments.} We acknowledge valuable comments by Carolin Kreisbeck and Giorgio Stefani.


\begin{thebibliography}{99}

\bibitem{ABCS} R. Alicandro, A. Braides, M. Cicalese, and M. Solci.
{\em Discrete Variational Problems with Interfaces}. Cambridge University Press, 2023.

\bibitem{BMC} J.C. Bellido,  J. Cueto, and C. Mora-Corral. $\Gamma$-convergence of polyconvex functionals involving $s$-fractional gradients to their local counterparts. Calc. Var. Partial Differential Equations 60 (2021), 7.

\bibitem{BBM} J. Bourgain, H. Brezis, and P. Mironescu. 
Another look at Sobolev spaces. 
In {\em Optimal Control and Partial Differential Equations}, IOS Press, Amsterdam, 2001, pp. 439--455.

\bibitem{BCS} E. Bru\`e, M. Calzi, G. E. Comi, and G. Stefani.
A distributional approach to fractional Sobolev spaces and fractional variation: asymptotics II.
C. R. Math. 360 (2022), 589--626 



\bibitem{CS1} G. E. Comi and G. Stefani, A distributional approach to fractional Sobolev spaces and fractional variation: existence of blow-up, J. Funct. Anal. 277 (2019), 3373--3435.

\bibitem{CS2} G.E. Comi and G. Stefani, A distributional approach to fractional Sobolev spaces and fractional variation: asymptotics I, Revista Matem\'atica Complutense, 2022.

\bibitem{de1}
M. D'Elia, Q. Du, C. Glusa, M. Gunzburger, X. Tian, and Z. Zhou.
Numerical methods for nonlocal and fractional models
Acta Numerica 29 (2020), 1--124

\bibitem{de2} M. D'Elia and M. Gunzburger.
The fractional Laplacian operator on bounded domains as a special case of the nonlocal diffusion operator.
Computers \& Mathematics with Applications 66 (2013), 1245--1260

\bibitem{DGLZ} Q. Du, M.Gunzburger, R.B. Lehoucq, and  K. Zhou.
A non-local vector calculus, non-local volume-constrained problems, and non-local balance laws.
Math. Mod. Meth. Appl. Sci. 23 (2013), 493--540

\bibitem{Fejer} L. Fej\'er, {\"U}ber die Positivit{\"a}t von Summen, die nach trigonometrischen oder Legendreschen Funktionen fortschreiten : Erste Mitteilung. Acta litterarum ac scientiarum Regiae Universitatis Hungaricae Francisco-Josephinae: Sectio scientiarum mathematicarum, (2). pp. 75-86. (1926)

\bibitem{Gray} R.M. Gray. {\em Toeplitz and Circulant Matrices: a Review}. Information Systems Laboratory, Stanford University, 2002.

\bibitem{KS} C. Kreisbeck and  H. Sch\"onberger. Quasiconvexity in the fractional calculus of variations: characterization of lower semicontinuity and relaxation. Nonlinear Anal. 215 (2022), 112625.

\bibitem{MS}
T. Mengesha and D. Spector.
Localization of nonlocal gradients in various topologies.
Calc. Var. Partial Differential Equations 52 (2015), 253--279

\bibitem{MMR} G.V. Milovanovi\'c, D.S. Mitrinov\'c, and Th.M. Rassias. {\em Topics in Polynomials: Extremal Problems,
Inequalities, Zeros}. World Sci., Singapore (1994)

\bibitem{rod} J.F. Rodrigues and L. Santos, On nonlocal variational and quasi-variational inequalities with fractional gradient, Appl. Math. Optim. 80 (2019), 835--852.

\bibitem{SS} T.-T. Shieh and  D.E. Spector. On a new class of fractional partial differential equations. Adv. Calc. Var. 8 (2015), 321--336

\bibitem{TD1} X. Tian and Q. Du.
Analysis and comparison of different approximations to nonlocal diffusion and linear peridynamic equations.
SIAM J. Numer. Anal. 51 (2013), 3458--3482

\bibitem{TD2} X. Tian and Q. Du.
Asymptotically compatible schemes and applications to robust discretization of nonlocal models.
SIAM J. Numer. Anal. 52 (2014), 1641--1665
\end{thebibliography}
\end{document}